\documentclass[11pt]{amsart}
\usepackage{a4wide}
\usepackage{amsmath}
\usepackage[utf8]{inputenc}
\usepackage{amssymb}
\usepackage{amsopn}
\usepackage{epsfig}
\usepackage{amsfonts}
\usepackage{latexsym}
\usepackage{graphicx}
\usepackage{enumerate}
\usepackage{mathrsfs}
\usepackage[colorlinks,linkcolor=blue,anchorcolor=blue,citecolor=blue]{hyperref}
\usepackage{tikz}

\newtheorem{theorem}{Theorem}[section]
\newtheorem{lemma}[theorem]{Lemma}
\newtheorem{proposition}[theorem]{Proposition}
\newtheorem{corollary}[theorem]{Corollary}

\theoremstyle{definition}

\newtheorem{question}[theorem]{Question}
\theoremstyle{remark}
\newtheorem{remark}[theorem]{Remark}

\numberwithin{equation}{section}
\counterwithin{figure}{section}

\newcommand{\R}{\ensuremath{\mathbb{R}}}
\newcommand{\N}{\ensuremath{\mathbb{N}}}

\newcommand{\F}{\mathcal F}

\renewcommand{\i}{\mathbf{i}}
\renewcommand{\v}{\mathbf{v}}
\renewcommand{\u}{\mathbf{u}}

\newcommand{\set}[1]{\left\{#1\right\}}

\newcommand{\ep}{\varepsilon}
\newcommand{\f}{\infty}

\newcommand{\si}{\sigma}

\newcommand{\diam}{\mathrm{diam}}

\newcommand{\ii}{\mathbf{i}}
\newcommand{\jj}{\mathbf{j}}

\newcommand{\p}{\boldsymbol{p}}
\renewcommand{\u}{\boldsymbol{u}}

\begin{document}
	
\allowdisplaybreaks
	
\title{Asymptotic separation of periodic orbits and fractal dimension}

\author[D. Kong]{Derong Kong}
\address[D. Kong]{College of Mathematics and Statistics, Center of Mathematics,  Chongqing University, Chongqing 401331, People's Republic of China.}
\email{derongkong@126.com}

\author[Z. Wang]{Zhiqiang Wang}
\address[Z. Wang]{College of Mathematics and Statistics, Center of Mathematics, Chongqing University, Chongqing 401331, People's Republic of China \;\&\; Department of Mathematics, University of British Columbia, Vancouver, British Columbia, V6T 1Z2, Canada}
\email{zhiqiangwzy@163.com,~zqwangmath@cqu.edu.cn}

\author[D. Yu]{Daohua Yu}
\address[D. Yu]{College of Mathematics and Statistics, Key Laboratory of Nonlinear Analysis and Its Application (Ministry of Education), Chongqing University, Chongqing 401331, People's Republic of China.}
\email{yudh@cqu.edu.cn}

\date{\today}


\begin{abstract}
The problem considered in this paper is related to Smale's seventh problem.
For a totally bounded metric dynamical system $(X, d, T)$ we introduce a new critical value $\mathfrak s(X, d, T)$ which quantifies the asymptotic separation of periodic orbits.
More precisely, $\mathfrak s(X,d,T)$ is defined to be the supremum of all $s\ge 0$ for which there exists a sequence of periodic orbits $\{\mathcal O_k\}_{k=1}^\infty$ in $(X,d,T)$ such that
\[
\lim_{k\to \infty} \#\mathcal O_k = +\infty \quad \text{and} \quad \liminf_{k\to\infty}\#\mathcal O_k\cdot\eta(\mathcal O_k)^s >0,
\]
where $\eta(\mathcal O_k)$ denotes the smallest distance between distinct points in $\mathcal O_k$.
When the dynamical system is induced by a self-similar iterated function system $\mathcal F$ satisfying the strong separation condition, we prove that this critical value is equal to the Hausdorff dimension $s_{\mathcal F}$ of its self-similar attractor. Furthermore, by using the relative entropy method we show that the associated empirical periodic measures converge weakly to the normalized $s_{\mathcal F}$-dimensional Hausdorff measure on the self-similar attractor.
\end{abstract}


\maketitle

\section{Introduction}\label{sec:Introduction}
Smale's seventh problem asks for $N$ distinct points $x_1,x_2,\dots,x_N$ on the $2$-sphere such that the energy
\[
\sum_{1\le i<j\le N}\log\frac{1}{\|x_i-x_j\|}
\]
is close to its infimum. The problem was formulated by Shub and Smale in their study of efficient homotopy algorithms for solving polynomial equations \cite{Smale-1998}. Recent progress on this problem was reported in 2021 \cite[Subsection 1.3]{BEMO-2021}. In 1994, \cite[Theorem 3.4]{RSZ-1994} established that any energy-minimizing configuration must be well-separated. From a dynamical viewpoint, if the points form a single periodic orbit, then the full orbit is determined by any one of its points. This motivates us to investigate whether there exists a sequence of well-separated periodic orbits.

Periodic points and periodic measures are two fundamental objects in dynamical systems.
Despite their apparent simplicity, they play an indispensable role in the study of dynamical systems.
For example, for topological dynamical systems with the specification property, periodic measures are dense in the space of invariant measures under the weak topology, see Sigmund \cite{Sigmund-1974}.
The growth rate of the number of periodic points is closely related to the topological entropy, see \cite{Bowen-1971,Katok-1980,Burguet-2020}.
Periodic points and periodic measures also play a central role in ergodic optimization, see the survey \cite{Jenkinson-2006,Jenkinson-2019}.

Recently, Gan and the third author \cite{Yu-Gan-2025} showed that, for any ergodic endomorphism on the $d$-torus, there exists a sequence of periodic orbits $\{\mathcal{O}_k\}_{k=1}^\f$ such that
\begin{equation}\label{result-Yu-Gan}
  \lim_{k \to +\f} \# \mathcal{O}_k = +\f \quad \text{and} \quad \liminf_{k \to +\f} \# \mathcal{O}_k \cdot   \eta( \mathcal{O}_k )^d >0,
\end{equation}
where $\#A$ denotes the cardinality of a set $A$ and $\eta(\mathcal O_k)$ denotes the smallest distance between distinct points in $\mathcal O_k$.
Furthermore, they proved that if a sequence of periodic orbits $\{\mathcal{O}_k\}_{k=1}^\f$ satisfies the conditions in (\ref{result-Yu-Gan}), then the corresponding sequence of periodic measures $\{\mu_k\}_{k=1}^\f$ converges weakly to the Lebesgue measure on the $d$-torus.
We emphasize that the exponent $d$ in (\ref{result-Yu-Gan}) is the dimension of the ambient space.
In the present paper, we attempt to investigate analogous results in the fractal setting.

Another motivation comes from the work of Boshernitzan \cite{Boshernitzan-1993}, in which he studied the quantitative recurrence rate of a metric measure preserving dynamical system $(X,d,\mu,T)$. In fact, he proved that for $\mu$-almost every $x\in X$,
\[
\liminf_{n\to+\f}n\cdot   d(T^n x, x)^\alpha<+\f,
\]
where $\alpha$ is the Hausdorff dimension of $X$. Instead of looking at the recurrence rate of typical points, we study the asymptotic separation of a sequence of periodic orbits $\{\mathcal O_k\}_{k=1}^\f$ through the quantities $\#\mathcal O_k\cdot\eta(\mathcal O_k)^{s}, \; s\ge 0$.

When the dynamical system is induced by a self-similar iterated function system $\mathcal F$ satisfying the strong separation condition, we show in Theorem \ref{thm:critical-value} that the Hausdorff dimension $s_{\F}$ of the self-similar attractor is the supremum of all $s\ge 0$ for which there exists a sequence of periodic orbits $\{ \mathcal O_k \}_{k=1}^\f$ such that
\[
\lim_{k \to +\f} \# \mathcal{O}_k = +\f \quad \text{and} \quad  \liminf_{k\to+\f}\#\mathcal O_k\cdot\eta(\mathcal O_k)^{s}>0.
\]
Moreover, we prove in Theorem \ref{thm:periodic-measure} that for any sequence of periodic orbits $\{ \mathcal O_k \}_{k=1}^\f$ satisfying
\[
\lim_{k \to +\f} \# \mathcal{O}_k = +\f \quad \text{and} \quad  \liminf_{k\to+\f} \frac{\log \#\mathcal O_k}{- \log \eta(\mathcal O_k)} \ge s_{\mathcal F},
\]
the corresponding sequence of periodic measures $\{\mu_k\}_{k=1}^\f$ converges weakly to normalized restriction of $s_{\F}$-dimensional Hausdorff measure on the self-similar attractor.
The proof is based on relative entropy estimates, where the invariance of these periodic measures plays an essential role.

\subsection{A new critical value}

We first introduce a new critical value concerning the asymptotic separation of periodic orbits for a metric dynamical system.

Let $(X,d)$ be a totally bounded metric space, i.e., for any $\ep>0$ there exist finitely many points $x_1, x_2, \ldots, x_N \in X$ such that
\[ X = \bigcup_{j=1}^N B(x_j, \ep), \]
where $B(x,r) : = \big\{ y \in X: d(y,x) < r \big\}$ is the open ball in $X$ with center $x$ and radius $r$.
Note that a compact metric space is totally bounded, and any subset of a totally bounded metric space is totally bounded.
Let $T: X \to X$ be a self-mapping.
The triple $(X,d,T)$ will be called a \emph{metric dynamical system}.

Given a metric dynamical system $(X, d, T)$, the \emph{orbit} of $x \in X$ is given by \[ \mathcal{O}_x := \big\{ T^n x: n \in \N\cup\{0\} \big\}. \]
A point $x \in X$ is called a \emph{periodic point} if there exists $n \in \N$ such that $T^n x = x$, and the integer $n$ is called a \emph{period} of $x$.
If $x\in X$ is a periodic point, then $\# \mathcal{O}_x$ is the least period of $x$.
Note that any period of $x$ is a multiple of $\# \mathcal{O}_x$.
For $A \subset X$, define \[ \eta(A) := \inf\big\{ d(x,y): x\ne y \in A \big\}, \] with the convention that $\inf \emptyset = +\f$.

Motivated by the result in (\ref{result-Yu-Gan}), we define a quantity that describes the asymptotic separation of periodic orbits.
For $s \ge 0$ and $n \in \N$, define
\begin{equation}\label{eq:P-n} \mathfrak{P}_n^s(X,d,T) := \sup\big\{ \# \mathcal{O}_x \cdot  \eta( \mathcal{O}_x )^s: \;\text{$x\in X$ is a periodic point with $\# \mathcal{O}_x \ge n$}  \big\}, \end{equation}
where we adopt the convention that $\sup \emptyset = 0$.
Note that $\mathfrak{P}_n^s(X,d,T)$ is decreasing as $n \to +\f$.
Thus, we can define
\begin{equation}\label{eq:P-def}
 \mathfrak{P}^s(X,d,T) := \lim_{n \to +\f} \mathfrak{P}_n^s(X,d,T) \in[0,+\f],
\end{equation}
which is called the \emph{$s$-dimensional periodic content} of $(X,d,T)$.

We will show in Lemma \ref{lem:critical-value}  that there exists a critical value of $s$ at which $\mathfrak{P}^s(X,d,T)$ jumps from $+\f$ to $0$ (see Figure \ref{figure-P-s} for an illustration).
We are interested in this critical value
\begin{equation}\label{eq:critical-value}
\mathfrak{s}(X,d,T):= \inf\big\{ s \ge 0: \mathfrak{P}^s(X,d,T) =0 \big\} = \sup \big\{ s \ge 0: \mathfrak{P}^s(X,d,T) = +\f \big\},
\end{equation}
and we call it  the \emph{periodic dimension} of $(X,d,T)$.
Indeed, we also have
\[ \mathfrak{s}(X,d,T)= \inf\big\{ s \ge 0: \mathfrak{P}^s(X,d,T) < + \f \big\} = \sup \big\{ s \ge 0: \mathfrak{P}^s(X,d,T) >0 \big\}. \]

\begin{figure}[htbp]
\centering
\begin{tikzpicture}
  \draw[->] (-0.3,0) -- (4,0) node[right] {$s$};
  \draw[->] (0,-0.3) -- (0,3.4);
  \draw[very thick] (0,3) -- (1.8,3); \draw[very thick] (1.8,0) -- (3.95,0);
  \node at (-0.15,-0.2) {$0$};
  \node[left] at (0,2) {$\mathfrak{P}^s(X,d,T)$};
  \node[left] at (0,3) {$+\infty$};
  \draw[dashed](1.8,0)--(1.8,3);
  \node at (1.8,-0.4) {$\mathfrak{s}(X,d,T)$};
\end{tikzpicture}
\caption{The graph of the periodic content $\mathfrak{P}^s(X,d,T)$, which jumps from infinity to zero at the periodic dimension $\mathfrak s(X, d, T)$.}
\label{figure-P-s}
\end{figure}
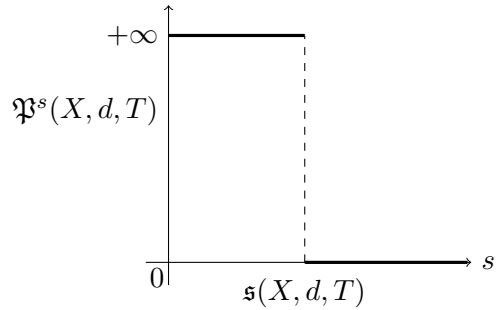

Note that the periodic dimension $\mathfrak{s}(X,d,T)$ depends on the metric $d$ of $X$. By the definition of periodic dimension, if the periods of periodic points in a metric dynamical system $(X,d,T)$ are bounded, then we have $\mathfrak{s}(X,d,T) =0$, but the converse is generally false.
We will present some properties of periodic content $\mathfrak P^s(X, d, T)$ and period dimension $\mathfrak s(X, d, T)$ in Section \ref{sec:property}.

\subsection{Self-similar iterated function system}
Next, we focus on dynamical systems induced by self-similar iterated function systems in fractal geometry.

A self-similar \emph{iterated function system} (IFS) on $\R^d$ is a finite collection of contractive similitudes, i.e., $\F=\big\{f_i(x)=r_i O_i x+b_i \big\}_{i=1}^m$, where each $0< r_i<1$, each $b_i\in\R^d$ and each $O_i$ is a $d\times d$ orthogonal matrix.
According to Hutchinson \cite{Hutchinson_1981}, there exists a unique \emph{self-similar attractor}, i.e., a unique non-empty compact set $E_{\F} \subset \R^d$ such that
\begin{equation}\label{eq:self-similar-set}
  E_\F = \bigcup_{i=1}^m f_i(E_\F).
\end{equation}
If the unions on the right-hand side of (\ref{eq:self-similar-set}) are pairwise disjoint, then the IFS $\mathcal{F}$ is said to satisfy the \emph{strong separation condition} (SSC).
In this case, we can define the inverse expanding mapping of $\mathcal{F}$. That is, $T_\F: E_\F \to E_\F$ is given by \[ T_\F|_{f_i(E_\F)} := f_i^{-1}. \]
Together with the Euclidean metric $|\cdot|$, we obtain a metric dynamical system $(E_\F, |\cdot|, T_\F)$.

Given a probability vector $\p=(p_1, \ldots, p_m)$, there exists a unique Borel probability measure $\mu_{\p}$ on $\R^d$, which is called the \emph{self-similar measure}, such that
\[ \mu_{\p} = \sum_{i=1}^{m} p_i \mu_{\p} \circ f_i^{-1}. \]

Associated with a self-similar IFS $\F$, the similarity dimension $s_\F$ is defined to be the unique non-negative root of the equation \[ \sum_{i=1}^{m} r_i^s = 1. \]
If $\F$ satisfies the SSC, then the Hausdorff dimension and the box-counting dimension of $E_\F$ coincide and are both equal to $s_\F$. Furthermore, we have $0< \mathcal{H}^{s_\F}(E_\F) < + \f$.
Here, $\mathcal{H}^s$ denotes the $s$-dimensional Hausdorff measure.
If we further take $\p=(r_1^{s_\F}, r_2^{s_\F}, \ldots, r_m^{s_\F})$, then the corresponding self-similar measure $\mu_{\p}$ coincides with the normalized restriction of $\mathcal{H}^{s_\F}$ on $E_\F$. We refer the reader to \cite{Falconer-2014-book} for an introduction of fractal geometry.

Our first main result establishes that the periodic dimension $\mathfrak{s}(E_\F, |\cdot|, T_\F)$ coincides with the similarity dimension $s_\F$.

\begin{theorem}\label{thm:critical-value}
  Let $\F=\big\{f_i(x)=r_i O_i x+b_i \big\}_{i=1}^m$ be a self-similar IFS on $\R^d$ that satisfies the SSC.
  Then we have  \[\mathfrak{s}(E_\F, |\cdot|, T_\F)=s_\F.\]
  Furthermore, for any sequence of periodic orbits $\{\mathcal O_k\}_{k=1}^\f$ in $(E_\F, |\cdot|, T_\F)$ we have
  \[
  \limsup_{k\to+\f}\#\mathcal O_k\cdot\eta(\mathcal O_k)^{s_\F}<+\f.
  \]
\end{theorem}
\begin{remark}\label{rem:critical-value}
 Theorem \ref{thm:critical-value}, together with Theorem  \ref{thm:P-s} (see below), implies that (i) if $s>s_\F$ then for any sequence of periodic orbits $\{\mathcal O_k\}_{k=1}^\f$ in $(E_\F, |\cdot|, T_\F)$ with $\#\mathcal O_k \to +\f$ we have
 \[
 \lim_{k\to+\f}\#\mathcal O_k\cdot\eta(\mathcal O_k)^s=0;
 \]
 (ii) if $s<s_\F$ then we can find a sequence of periodic orbits $\{\mathcal O_k\}_{k=1}^\f$ in $(E_\F, |\cdot|, T_\F)$ such that
 \[
 \lim_{k\to+\f}\#\mathcal O_k = +\f \quad \text{and} \quad\lim_{k\to+\f}\#\mathcal O_k\cdot\eta(\mathcal O_k)^s=+\f.
 \]
 \end{remark}

Our second main result concerns the weak limits of periodic measures.
Let $\delta_a$ denote the Dirac measure supported on a point $a \in \R^d$.

\begin{theorem}\label{thm:periodic-measure}
  Suppose that $\F=\big\{f_i(x)=r_i O_i x+b_i \big\}_{i=1}^m$ is a self-similar IFS on $\R^d$ that satisfies the SSC.
  In the metric dynamical system $(E_{\F}, |\cdot|, T_{\F})$, let $\{x_n\}_{n=1}^\f$ be a sequence of periodic points with $\#\mathcal{O}_{x_n} \to +\f$ as $n \to +\f$.
  Define \[ \mu_n := \frac{1}{\#\mathcal{O}_{x_n}} \sum_{y \in \mathcal{O}_{x_n}} \delta_y. \]
  If
  \begin{equation}\label{eq:s-F-lower-bound}
    \liminf_{n \to +\f} \frac{\log \#\mathcal O_{x_n}}{- \log \eta(\mathcal O_{x_n})} \ge s_{\mathcal F},
  \end{equation}
  then the sequence $\{\mu_n\}_{n=1}^\f$ converges weakly to $\mu_{\p}$ with $\p=(r_1^{s_\F}, r_2^{s_\F}, \ldots, r_m^{s_\F})$.
\end{theorem}
\begin{remark}
(a) Note that \[ \limsup_{n \to +\f} \frac{\log \#\mathcal O_{x_n}}{- \log \eta(\mathcal O_{x_n})} \le \overline{\dim}_{\mathrm{B}} E_{\F} = s_{\mathcal F}, \]
where $\overline{\dim}_{\mathrm{B}}$ denotes the upper box-counting dimension, and the last equality follows from \cite[Corollary 3.3]{Falconer_1997}.
So, the condition (\ref{eq:s-F-lower-bound}) is equivalent to that \[  \lim_{n \to +\f} \frac{\log \#\mathcal O_n}{- \log \eta(\mathcal O_n)} = s_{\mathcal F}. \]
We don't know whether the condition (\ref{eq:s-F-lower-bound}) is necessary.

(b) For any $n \in \N$, by Remark \ref{rem:critical-value}, we can find a periodic orbit $\mathcal{O}_n$ such that $\#\mathcal{O}_n \ge n$ and \[ \#\mathcal O_n \cdot\eta(\mathcal O_n)^{s_{\F} - \frac{1}{n}} \ge 1.  \]
It follows that \[ \liminf_{n \to +\f} \frac{\log \#\mathcal O_{n}}{- \log \eta(\mathcal O_{n})} \ge s_{\mathcal F}. \]
Thus, Theorem \ref{thm:critical-value} guarantees the existence of a sequence of periodic orbits satisfying the condition (\ref{eq:s-F-lower-bound}).
\end{remark}

In contrast, we recall two related results concerning the approximation of self-similar measures.
Suppose that $\F=\big\{f_i(x)=r_i O_i x+b_i \big\}_{i=1}^m$ is a self-similar IFS on $\R^d$ that satisfies the SSC and let $E_{\F}$ be its self-similar attractor.
Let $\mu_{\p}$ be the self-similar measure for $\p=(r_1^{s_\F}, r_2^{s_\F}, \ldots, r_m^{s_\F})$.
For $\ep >0$, let $A_\ep \subset E_{\F}$ with $\eta(A_\ep)\ge \ep$ have the maximum cardinality.
Lalley \cite{Lalley-1988} showed that
\begin{equation}\label{eq:Lalley-result}
\frac{1}{\# A_\ep} \sum_{x \in A_\ep} \delta_x \longrightarrow \mu_{\p}\quad\text{as}\;\; \ep \to 0.
\end{equation}
The other result involves asymptotically best-packing configurations.
For $n \ge 2$, define $\eta(E_\F, n) := \sup\{ \eta(A): A \subset E_{\F},\; \# A = n \}$.
A sequence $\{A_n\}$ of finite subsets of $E_{\F}$ with $\#A_n = n$ is called \emph{asymptotically best-packing} if \[ \lim_{n \to +\f} \frac{\eta(A_n)}{\eta(E_{\F},n)} =1. \]
Borodachov \cite{Borodachov-2012} proved that, assuming the limit $n^{1/s_{\F}} \cdot \eta(E_{\F}, n)$ exists as $n \to +\f$, for any asymptotically best-packing sequence $\{A_n\}$,
\begin{equation}\label{eq:Borodachov-result}
\frac{1}{\# A_n} \sum_{x \in A_n} \delta_x \longrightarrow \mu_{\p}\quad\text{as}\;\; n \to +\f.
\end{equation}
A sufficient and necessary condition for the existence of the limit $n^{1/s_{\F}} \cdot\eta(E_{\F}, n) $ as $n \to +\f$ is given in \cite{Lalley-1988,Anderson-Reznikov-2024}.
Note that all $\mu_n$ in Theorem \ref{thm:periodic-measure} are $T_{\F}$-invariant, while the measures in (\ref{eq:Lalley-result}) and (\ref{eq:Borodachov-result}) are generally not $T_{\F}$-invariant.
The above two results also hold under the open set condition; we refer the reader to \cite{Lalley-1988,Borodachov-2012} for further details.

The rest of the paper is organized as follows.
In Section \ref{sec:property}, we present some properties of periodic content and dimension.
We will prove Theorem \ref{thm:critical-value} and Theorem \ref{thm:periodic-measure} in Section \ref{sec:self-similar-IFS} and Section \ref{sec:self-similar-measure}, respectively.
Finally, we give some remarks in Section \ref{sec:remarks}.

\section{Properties of periodic content and dimension}\label{sec:property}

In this section, let $(X,d,T)$ be a given metric dynamical system with $(X, d)$ a totally bounded metric space, and we will give some properties of periodic content $\mathfrak{P}^s(X,d,T)$ and periodic dimension $\mathfrak{s}(X, d,T)$ as defined in (\ref{eq:P-def}) and (\ref{eq:critical-value}), respectively.
We first establish the existence of periodic dimension.

\begin{lemma}\label{lem:critical-value}
  Let $s \ge 0$. If $\mathfrak{P}^s(X,d,T)< + \f$, then for any $t > s$, we have $\mathfrak{P}^t(X,d,T) =0$.
\end{lemma}
\begin{proof}
  Fix $\ep > 0$.
  Since $X$ is totally bounded, we can find finitely many points $x_1, x_2, \ldots, x_N \in X$ such that
  \[ X = \bigcup_{j=1}^N B(x_j, \ep/2). \]
  For $n > N$, if $x$ is a periodic point with $\# \mathcal{O}_x \ge n$, then by the pigeonhole principle there exists $1\le j \le N$ such that \[ \# \big( B(x_j, \ep/2) \cap \mathcal{O}_x \big) \ge 2, \]
  which implies that $\eta(\mathcal{O}_x) \le \ep$. This means that for $t>s$,
  \[ \# \mathcal{O}_x \cdot  \eta( \mathcal{O}_x )^t \le \ep^{t-s} \cdot \# \mathcal{O}_x \cdot  \eta( \mathcal{O}_x )^s. \]
  Thus, by (\ref{eq:P-n}) we obtain that for any $n > N$,
  \[ \mathfrak{P}_n^t(X,d,T) \le \ep^{t-s} \cdot \mathfrak{P}_n^s(X,d,T). \]
  Letting $n \to +\f$, it follows by (\ref{eq:P-def}) that \[ \mathfrak{P}^t(X,d,T) \le \ep^{t-s} \cdot \mathfrak{P}^s(X,d,T). \]
  Since $\ep>0$ is arbitrary and $\mathfrak{P}^s(X,d,T)<+\f$, we conclude that $\mathfrak{P}^t(X,d,T) =0$.
\end{proof}
By Lemma \ref{lem:critical-value} it follows that the periodic dimension $\mathfrak{s}(X,d,T)$ exists, which is given by
\begin{align*}
 \mathfrak{s}(X,d,T)&= \inf\big\{ s \ge 0: \mathfrak{P}^s(X,d,T) =0 \big\} = \sup \big\{ s \ge 0: \mathfrak{P}^s(X,d,T) = +\f \big\}\\
    &=\inf\big\{ s \ge 0: \mathfrak{P}^s(X,d,T) < + \f \big\} = \sup \big\{ s \ge 0: \mathfrak{P}^s(X,d,T) >0 \big\}.
\end{align*}
Let $\Gamma(X)$ denote the set of all sequences $\set{x_n}\subset X$ of periodic points with $\# \mathcal{O}_{x_n} \to +\f$ as $n \to +\f$.
Next, we give a lower bound on periodic content $\mathfrak{P}^s(X,d,T)$.

\begin{lemma}\label{lem:lower-bound-P-s}
  Let $s\ge 0$. For any sequence $\{x_n\} \in \Gamma(X)$, we have
  \[ \mathfrak{P}^s(X,d,T) \ge \limsup_{n \to +\f} \# \mathcal{O}_{x_n} \cdot   \eta( \mathcal{O}_{x_n} )^s. \]
\end{lemma}
\begin{proof}
  Fix $k \in \N$, and take $\set{x_n}\in\Gamma(X)$. Since $\# \mathcal{O}_{x_n} \to +\f$ as $n\to+\f$, there exists $n_0\in \N$ such that $\# \mathcal{O}_{x_n} > k$ for all $n \ge n_0$. By (\ref{eq:P-n}) it follows that
  \[ \mathfrak{P}_k^s(X,d,T) \ge \# \mathcal{O}_{x_n} \cdot  \eta( \mathcal{O}_{x_n} )^s  \qquad \forall n \ge n_0, \]
  which implies that \[ \mathfrak{P}_k^s(X,d,T) \ge \limsup_{n \to +\f} \# \mathcal{O}_{x_n} \cdot \eta( \mathcal{O}_{x_n} )^s. \]
  This completes the proof by letting $k \to +\f$.
\end{proof}

The following result provides a characterization of periodic content $\mathfrak{P}^s(X,d,T)$, which can be used to obtain the lower and upper bounds for periodic dimension $\mathfrak{s}(X,d,T)$.

\begin{theorem}\label{thm:P-s}
  Let $s\ge 0$. The following statements on $\mathfrak{P}^s(X,d,T)$ hold.
  \begin{enumerate}[{\rm(i)}]
    \item $\mathfrak{P}^s(X,d,T)>0$ if and only if there exists a sequence $\{x_n\} \in \Gamma(X)$, \[\liminf_{n \to +\f} \# \mathcal{O}_{x_n} \cdot \eta( \mathcal{O}_{x_n} )^s >0. \]
    \item $\mathfrak{P}^s(X,d,T)<+\f$ if and only if for any sequence $\{x_n\} \in \Gamma(X)$,  \[ \limsup_{n \to +\f} \# \mathcal{O}_{x_n} \cdot \eta( \mathcal{O}_{x_n} )^s <+\f. \]
  \end{enumerate}
\end{theorem}
\begin{proof}
  (i) The sufficiency follows directly from Lemma \ref{lem:lower-bound-P-s}.
  For the necessity, we assume that $\mathfrak{P}^s(X,d,T)>0$.
  Then by (\ref{eq:P-def}) there exist $\ep_0>0$ and $n_0 \in \N$ such that $\mathfrak{P}_n^s(X,d,T)> \ep_0$ for all $n \ge n_0$.
  For any $n \ge n_0$, by (\ref{eq:P-n}) we can choose a periodic point $x_n$ such that
  \[ \# \mathcal{O}_{x_n}\ge n \quad \text{and} \quad  \# \mathcal{O}_{x_n} \cdot  \eta( \mathcal{O}_{x_n} )^s \ge \ep_0. \]
  Thus, we obtain that $\{x_n\} \in \Gamma(X)$ and \[ \liminf_{n \to +\f} \# \mathcal{O}_{x_n} \cdot   \eta( \mathcal{O}_{x_n} )  ^s \ge \ep_0, \]
  as desired.

  (ii) The necessity follows directly from Lemma \ref{lem:lower-bound-P-s}.
  Next, we prove the sufficiency.
  Assume on the contrary that $\mathfrak{P}^s(X,d,T)=+\f$.
  Note that $\mathfrak{P}_n^s(X,d,T)$ is decreasing as $n \to +\f$.
  So we have $\mathfrak{P}_n^s(X,d,T)=+\f$ for all $n \in \N$.
  For any $n \in \N$, by (\ref{eq:P-n}) we can choose a periodic point $x_n$ such that
  \[ \# \mathcal{O}_{x_n}\ge n \quad \text{and} \quad  \# \mathcal{O}_{x_n} \cdot  \eta( \mathcal{O}_{x_n} )  ^s \ge n. \]
  We clearly have $\{x_n\} \in \Gamma(X)$, but
  \[ \lim_{n \to +\f} \# \mathcal{O}_{x_n} \cdot   \eta( \mathcal{O}_{x_n} )  ^s = +\f, \]
  leading to a contradiction.
  This completes the proof.
\end{proof}

Next, we show that the periodic dimension is invariant under the bi-Lipschitz conjugate.
Let $(X,d_X)$ and $(Y,d_Y)$ be two metric spaces.
A mapping $\phi: X \to Y$ is said to be \emph{Lipschitz} if there exists a constant $C>0$, which is called a \emph{Lipschitz constant} of $\phi$, such that
\[ d_Y\big( \phi(x_1), \phi(x_2) \big) \le C \cdot d_X(x_1,x_2) \qquad \forall x_1,x_2 \in X.  \]
More generally, for $\alpha>0$, a mapping $\phi: X \to Y$ is said to be \emph{$\alpha$-H\"{o}lder} if there exists a constant $C>0$ such that
\[ d_Y\big( \phi(x_1), \phi(x_2) \big) \le C \cdot \big( d_X(x_1,x_2) \big)^{\alpha} \qquad \forall x_1,x_2 \in X.  \]
Clearly, every Lipschitz mapping is $1$-H\"{o}lder.
If $\phi:X \to Y$ is a bijective mapping and both $\phi$ and $\phi^{-1}$ are Lipschitz, then we say that $\phi$ is \emph{bi-Lipschitz}.

\begin{proposition}
  Let $(X,d_X, T)$ and $(Y, d_Y, S)$ be two metric dynamical systems. If there exists an $\alpha$-H\"{o}lder mapping $\phi: X \to Y$ such that $\phi\circ T = S \circ \phi$ and $\# \phi^{-1}(\{y\}) < +\f$ for any $y \in Y$,
  then we have
  \[\mathfrak{s}(\phi(X),d_Y,S) \le \frac{\mathfrak{s}(X,d_X,T)}{\alpha}.\]
\end{proposition}
\begin{proof}
  Note first that $S\big( \phi(X) \big) = \phi\big( T(X) \big) \subset \phi(X)$.
  So we can consider the metric dynamical system $(\phi(X), d_Y, S)$.

  Assume that $y \in \phi(X)$ is a periodic point in $(\phi(X),d_Y,S)$.
  Then there exists $n_0 \in \N$ such that $S^{n_0} y = y$.
  Take $x_0 \in X$ such that $\phi(x_0) = y$.
  Then for any $k \in \N$ we have $\phi(T^{kn_0} x_0) = S^{kn_0} \phi(x_0) = S^{kn_0} y = y$.
  This implies that $\{ T^{kn_0} x_0: k \in \N \}  \subset \phi^{-1}(\{y\})$, which is a finite set.
  Thus, there exist $k_1,k_2 \in \N$ with $k_1 < k_2$ such that $T^{k_1 n_0} x_0 = T^{k_2 n_0} x_0$.
  Let $x = T^{k_1 n_0} x_0$. Then we have $\phi(x) = y$ and $T^{(k_2 - k_1)n_0} x = x$.
  In other words, we find a periodic point $x$ in $(X,d_X, T)$ such that $\phi(x) = y$.
  Note that $\phi: X\to Y$ is $\alpha$-H\"{o}lder and $\phi\circ T=S\circ\phi$. Then  $\mathcal{O}_y = \phi(\mathcal{O}_x)$, and
  it follows that \[ \# \mathcal{O}_y \le \#\mathcal{O}_x \quad \text{and}\quad \eta(\mathcal{O}_y) \le C \cdot  \eta(\mathcal{O}_x)  ^\alpha, \]
  where $C>0$ is the constant in the definition of $\alpha$-H\"{o}lder mapping. So
  for $s \ge 0$ we have
  \[
  \#\mathcal O_y  \cdot  \eta(\mathcal O_y) ^{s/\alpha}\le \#\mathcal O_x\cdot C^{s/\alpha}\cdot\eta(\mathcal O_x)^s.
  \] This together with (\ref{eq:P-n}) implies that \[ \mathfrak{P}_n^{s/\alpha}(\phi(X),d_Y,S) \le C^{s/\alpha} \cdot \mathfrak{P}_n^s(X,d_X,T)\qquad \forall n \in \N. \]
  Letting $n \to +\f$ yields \[ \mathfrak{P}^{s/\alpha}(\phi(X),d_Y,S) \le C^{s/\alpha} \cdot \mathfrak{P}^s(X,d_X,T). \]
  Therefore, we conclude that $\mathfrak{s}(\phi(X),d_Y,S) \le \mathfrak{s}(X,d_X,T)/\alpha$.
\end{proof}

\begin{corollary}\label{cor:bi-Lipschitz}
  Let $(X,d_X, T)$ and $(Y, d_Y, S)$ be two metric dynamical systems. If there exists a bi-Lipschitz mapping $\phi: X \to Y$ such that $\phi\circ T = S \circ \phi$,
  then we have $\mathfrak{s}(Y,d_Y,S) = \mathfrak{s}(X,d_X,T)$.
\end{corollary}

Finally, we show that the periodic dimension is bounded above by the upper box-counting dimension of $X$, which is defined by (cf.~\cite{Falconer-2014-book})
\[ \overline{\dim}_{\mathrm{B}} X := \limsup_{r \to 0^+} \frac{\log \mathcal{N}_r(X)}{-\log r}, \]
where $\mathcal{N}_r(X)$ denotes the minimum number of open balls of radius $r$ needed to cover $X$.

\begin{lemma}\label{lem:upper-bound}
  We have $\mathfrak{s}(X,d,T) \le \overline{\dim}_{\mathrm{B}} X$.
\end{lemma}
\begin{proof}
  The conclusion is immediate when $\overline{\dim}_{\mathrm{B}} X = +\f$.
  Hence, without loss of generality, we may assume that $\overline{\dim}_{\mathrm{B}} X < +\f$.
  Fix $s > \overline{\dim}_{\mathrm{B}} X$.
  By the definition of upper box-counting dimension, there exists $r_0 > 0$ such that
  \begin{equation}\label{eq:box-1} \mathcal{N}_r(X) \le r^{-s} \qquad \forall 0< r < r_0. \end{equation}
  Take any sequence $\{x_n\}\in \Gamma(X)$, and  write $\eta_n = \eta(\mathcal{O}_{x_n})$.
  Note that an open ball of radius $\eta_n/2$ intersects $\mathcal{O}_{x_n}$ at most one point.
  It follows that $\#\mathcal{O}_{x_n} \le \mathcal{N}_{\eta_n/2}(X)$.
  Since $X$ is totally bounded and $\#\mathcal O_{x_n}\to +\f$ as $n\to+\f$, we must have $\eta_n \to 0$ as $n \to +\f$.
  Thus, by (\ref{eq:box-1}) it follows that for all sufficiently large $n \in \N$, we have
  \[ \#\mathcal{O}_{x_n} \le \mathcal{N}_{\eta_n/2}(X) \le \Big( \frac{\eta_n}{2} \Big)^{-s}, \]
  i.e., $\#\mathcal{O}_{x_n} \cdot (\eta_n)^s \le 2^s$.
  It follows that \[ \limsup_{n \to +\f} \# \mathcal{O}_{x_n} \cdot   \eta( \mathcal{O}_{x_n} )  ^s \le 2^s.  \]
  By Theorem \ref{thm:P-s} (ii), we obtain that $\mathfrak{P}^s(X,d,T) < +\f$.
  So we have $\mathfrak{s}(X,d,T) \le s$.
  Since $s>\overline{\dim}_{\mathrm{B}} X$ was arbitrary, we conclude that $\mathfrak{s}(X,d,T) \le \overline{\dim}_{\mathrm{B}} X$.
\end{proof}

\section{Self-similar iterated function systems}\label{sec:self-similar-IFS}

In this section, we will prove Theorem \ref{thm:critical-value}.
We always assume that $\mathcal{F} =\{f_i(x) = r_i O_i x + b_i\}_{i=1}^m$ is a self-similar IFS on $\R^d$ satisfying the SSC.
Recall from Section \ref{sec:Introduction} the metric dynamical system $(E_{\mathcal F}, |\cdot|, T_{\mathcal F})$, where $E_{\mathcal F}$ is the self-similar attractor of $\mathcal F$ and $T_{\F}$ is the inverse expanding mapping of $\F$.
Recall that $s_{\F}$ is the similarity dimension of $\mathcal{F}$, which is the unique non-negative root of $\sum_{i=1}^{m}r_i^s=1$.
For brevity, write $\mathfrak{P}^s(\F)$ and $\mathfrak{s}(\F)$ for $\mathfrak{P}^s(E_\F, |\cdot|, T_\F)$ and $\mathfrak{s}(E_\F, |\cdot|, T_\F)$, respectively.
Since $\mathcal F$ satisfies the SSC, it is well-known that (cf.~\cite{Falconer-2014-book}) $\dim_{\mathrm H} E_{\mathcal F}=\overline{\dim}_{\mathrm B} E_{\mathcal F}=s_{\F}$.
By Lemma \ref{lem:upper-bound}, we have $\mathfrak{s}(\F) \le \overline{\dim}_{\mathrm{B}} E_{\F} = s_{\F}$. In fact, we have the following stronger result.

\begin{lemma}\label{lem:finite-P}
 We have $\mathfrak P^{s_\F}(\F)<+\f$.
\end{lemma}
\begin{proof}
  Write $s=s_\F$. Since $\F$ satisfies the SSC, we have $0< \mathcal{H}^s(E_{\F}) < +\f$ and the associated Hausdorff measure $\mathcal H^s|_{E_{\F}}$ is $s$-Ahlfors regular (cf. \cite{Hutchinson_1981}). So there exists a constant $C>1$ such that for any ball $B(x,r)$ with $x\in E_\F$ and $0< r \le \mathrm{\diam}E_{\F}$ we have
  \begin{equation}\label{eq:sep-1}
    C^{-1}\cdot r^s\le \mathcal H^s(E_\F\cap B(x,r))\le C\cdot r^s.
  \end{equation}

  For any periodic point $x$ in $(E_\F, |\cdot|, T_\F)$ whose least period is sufficiently large, note that the open balls $B(y, \frac{\eta(\mathcal O_x)}{2}), y\in\mathcal O_x$ are pairwise disjoint, and it follows from (\ref{eq:sep-1}) that
  \begin{align*}
    \#\mathcal O_x\cdot\eta(\mathcal O_x)^s & =2^s C\cdot \#\mathcal O_x\cdot C^{-1}\left(\frac{\eta(\mathcal O_x)}{2}\right)^s \\
     & \le 2^s C\sum_{y\in\mathcal O_x}\mathcal H^s\left(E_\F\cap B\Big(y,\frac{\eta(\mathcal O_x)}{2}\Big)\right)\\
     &=2^s C\cdot\mathcal H^s\left(E_\F\cap\bigcup_{y\in\mathcal O_x}B\Big(y,\frac{\eta(\mathcal O_x)}{2}\Big)\right)\\
     &\le 2^s C\cdot\mathcal H^s(E_\F).
  \end{align*}
  By (\ref{eq:P-n}) and (\ref{eq:P-def}) we conclude that  $\mathfrak P^s(\F)\le 2^s C\cdot\mathcal H^s(E_\F)<+\f$ as required.
\end{proof}

To prove Theorem \ref{thm:critical-value}, it remains to show that $\mathfrak{s}(\F) \ge s_{\F}$.
To this end, we introduce a symbolic dynamical system that is bi-Lipschitz conjugate to $(E_\F, |\cdot|, T_\F)$.
Let $\Sigma := \{ 1, 2, \ldots, m\}^\N$ be the symbolic space over the alphabet $\{1,2,\ldots, m\}$, and let $\sigma$ be the left shift on $\Sigma$, i.e., $\sigma(i_1 i_2 i_3 \ldots) = i_2 i_3 \ldots$.
For a finite word $i_1 i_2 \ldots i_n \in \{ 1, 2, \ldots, m\}^n$, write $f_{i_1 i_2 \ldots i_n} := f_{i_1} \circ f_{i_2} \circ \cdots \circ f_{i_n}$ and $r_{i_1 i_2 \ldots i_n} = r_{i_1} r_{i_2} \cdots r_{i_n}$.
For $\ii\ne \jj \in \Sigma$, let $\ii\wedge\jj$ denote the longest common prefix of $\ii$ and $\jj$.
We equip $\Sigma$ with the metric
\begin{equation}\label{eq:d-F}
  d_\F(\ii, \jj) := r_{\ii \wedge \jj}.
\end{equation}
Then $(\Sigma,d_\F)$ is a compact metric space, and $(\Sigma,d_\F,\sigma)$ forms a metric dynamical system.
There is a natural coding mapping $\pi: \Sigma \to E_{\F}$ defined by
\begin{equation}\label{eq:coding-map}
\pi(\ii) = \lim_{n \to +\f} f_{i_1 i_2 \ldots i_n}(0) \quad \text{for}\;\; \ii=(i_k)_{k=1}^\f \in \Sigma.
\end{equation}
Note that the point $0$ in the definition of $\pi$ can be replaced by any other point.

\begin{lemma}\label{lem:pi-bi-Lipschitz}
  The coding mapping $\pi$ is bi-Lipschitz, and $\pi \circ \sigma = T_{\F} \circ \pi$.
\end{lemma}
\begin{proof}
  Take $\ii=(i_k)_{k=1}^\f \in \Sigma$, and we have
  \[ \pi(\ii) = \lim_{n \to +\f} f_{i_1 i_2 \ldots i_n}(0) = f_{i_1} \Big( \lim_{n \to +\f} f_{i_2 \ldots i_n}(0) \Big) = f_{i_1} \Big( \pi\big( \sigma (\ii) \big) \Big). \]
  This means that $\pi(\ii) \in f_{i_1}(E_{\F})$. Since $\F$ satisfies the SSC, we have $\pi(\ii)\notin f_j(E_\F)$ for any $j\ne i_1$.
  Thus, we obtain that $T_{\F} \big( \pi(\ii) \big) = \pi\big(\sigma (\ii) \big)$.
  That is, $\pi \circ \sigma = T_{\F} \circ \pi$.
  Next, we show that $\pi$ is bi-Lipschitz.

  Take $\ii \ne \jj \in \Sigma$, and write $\ii = (i_k)_{k=1}^\f$ and $\jj = (j_k)_{k=1}^\f$.
  If $i_1 \ne j_1$, then $d_{\F}(\ii,\jj) = 1$, and furthermore,
  \begin{align*}
   |\pi(\ii) - \pi(\jj)| &\le \mathrm{diam}\; E_\F=\mathrm{diam}\; E_\F\cdot d_\F(\ii,\jj),\\
   |\pi(\ii) - \pi(\jj)| &\ge \mathrm{dist}\big( f_{i_1}(E_\F), f_{j_1}(E_\F) \big) \ge \min_{1\le i< j \le m} \mathrm{dist}\big( f_i(E_\F), f_j(E_\F) \big)\cdot d_\F(\ii,\jj).
   \end{align*}
  Otherwise, let $n = \max \{ k \in \N: i_\ell = j_\ell \;\;\forall 1 \le \ell \le k \}$.
  Then we have $d_{\F}(\ii,\jj) = r_{i_1 i_2 \ldots i_n}$.
  Note that $\pi(\ii), \pi(\jj) \in f_{i_1 i_2 \ldots i_n}(E_{\F})$.
  So we have \[ |\pi(\ii) - \pi(\jj)| \le \mathrm{diam}\; f_{i_1 i_2 \ldots i_n}(E_{\F}) = r_{i_1 i_2 \ldots i_n} \cdot \mathrm{diam}\; E_\F=\mathrm{diam}\; E_\F\cdot d_\F(\ii,\jj). \]
  On the other hand,
  \begin{align*}
    |\pi(\ii) - \pi(\jj)| & \ge \mathrm{dist}\big( f_{i_1 i_2 \ldots i_{n+1}}(E_\F), f_{j_1 j_2 \ldots j_{n+1}}(E_\F) \big)\\
       &= r_{i_1 i_2 \ldots i_n} \cdot \mathrm{dist}\big( f_{i_{n+1}}(E_\F), f_{j_{n+1}}(E_\F) \big) \\
    & \ge r_{i_1 i_2 \ldots i_n} \cdot \min_{1\le i< j \le m} \mathrm{dist}\big( f_i(E_\F), f_j(E_\F) \big)\\
    &=\min_{1\le i< j \le m} \mathrm{dist}\big( f_i(E_\F), f_j(E_\F) \big)\cdot d_\F(\ii,\jj).
  \end{align*}
  Let \[ C_{\F} := \max\bigg\{ \frac{1}{\displaystyle\min_{1\le i< j \le m} \mathrm{dist}\big( f_i(E_\F), f_j(E_\F) \big)},\; \mathrm{diam}\; E_\F \bigg\}. \]
  We conclude that
  \[ C_{\F}^{-1} \cdot d_{\F}(\ii,\jj) \le  |\pi(\ii) - \pi(\jj)| \le C_{\F}\cdot d_{\F}(\ii,\jj), \]
  i.e., $\pi$ is a bi-Lipschitz mapping.
\end{proof}

It follows from Lemma \ref{lem:pi-bi-Lipschitz} and Corollary  \ref{cor:bi-Lipschitz} that $ \mathfrak{s}(\F) = \mathfrak{s}(\Sigma, d_{\F}, \sigma)$.
In the following we only need to deal with the symbolic dynamical system $(\Sigma, d_{\F}, \sigma)$.
We first consider the homogeneous case.

\begin{proposition}\label{prop:homogeneous}
If $\F$ is homogeneous, i.e., $r_i = r$ for all $1 \le i \le m$, then for any $n \in \N$, there exists a periodic point $\ii^{(n)}$ in $(\Sigma, d_{\F}, \sigma)$ such that \[ \# \mathcal{O}_{\ii^{(n)}} = m^n \quad \text{and}\quad \eta(\mathcal{O}_{\ii^{(n)}}) = r^{n-1}. \]
\end{proposition}
\begin{proof}
  An algorithm to construct these periodic sequences $\ii^{(n)}$ can be found in \cite{Martin-1934}.
  In fact, $\ii^{(n)}$ is a de Bruijn sequence of order $n$ on a size-$m$ alphabet and can be constructed by using the de Bruijn graph \cite{de-Bruijn-1946}. For completeness, we present the detailed construction as follows.

  Since $\F$ is homogeneous, we have $d_{\F}(\ii,\jj) = r^{|\ii \wedge \jj|}$, where $|\ii\wedge\jj|$ denotes the length of the longest common prefix of $\ii$ and $\jj$.
  For $n = 1$, we can take $\ii^{(1)} = (1 2 \ldots m)^\f$.
  In the following, we assume that $n \ge 2$.

  We first construct a directed graph $G=(V,E)$, where $V = \{1,2,\ldots, m\}^{n-1}$.
  For $\u=u_1\ldots u_{n-1}, \v=v_1\ldots v_{n-1}\in V$ we draw a directed edge $(\u,\v)\in E$  from $\u$ to $\v$ if \[ u_2\ldots u_{n-1}=v_1\ldots v_{n-2},\] and furthermore, we label the edge $(\u,\v)$ by $\mathcal{L}(\u,\v)=u_1$.
  An infinite path in $G$ gives rise to an infinite sequence of labels, and clearly the sequence corresponding to a path starting at $\v \in V$ begins with $\v$.
  One can easily verify that the graph $G$ is strongly connected, and each vertex in $G$ has $m$ ingoing edges and $m$ outgoing edges.
  Note that the number of directed edges in $G$ is $m^n$.
  By \cite[Theorem 7.1]{Harary-1969} it follows that $G$ has a closed eulerian trail
  \begin{equation}\label{eq:eulerian-trail}
    \v_0 \xrightarrow{(\v_0,\v_1)} \v_1 \xrightarrow{(\v_1,\v_2)} \v_2 \cdots \cdots \v_{m^n-1} \xrightarrow{(\v_{m^n-1},\v_{m^n})} \v_{m^n} = \v_0
  \end{equation}
  which contains each directed edge in $G$ exactly once.
  Take $\ii^{(n)}$ to be the label sequence corresponding to this eulerian trail, i.e., \[ \ii^{(n)} = \big( \mathcal{L}(\v_0,\v_1) \mathcal{L}(\v_1,\v_2) \ldots\ldots \mathcal{L}(\v_{m^n-1},\v_0) \big)^\f. \]

  For any $n$-length word $i_1 i_2 \ldots i_n \in \{1,2,\ldots,m\}^n$, there exists a unique $0 \le k \le m^n -1$ such that the directed edge $(\v_k,\v_{k+1})$ in the eulerian trail (\ref{eq:eulerian-trail}) satisfies
  \[ \v_k = i_1 i_2 \ldots i_{n-1} \quad \text{and} \quad \v_{k+1} = i_2 i_3 \ldots i_n. \]
  Then the sequences $\sigma^k(\ii^{(n)})$ and $\sigma^{k+1}(\ii^{(n)})$ begin with $\v_k$ and $\v_{k+1}$, respectively. So we obtain that the sequence $\sigma^k(\ii^{(n)})$ begins with $i_1 i_2 \ldots i_n$.
  Thus, we conclude that \[ \# \mathcal{O}_{\ii^{(n)}} = m^n \quad \text{and}\quad \eta(\mathcal{O}_{\ii^{(n)}}) = r^{n-1}, \]
  as desired.
\end{proof}

A general self-similar IFS can be approximated by sub-homogeneous self-similar IFSs in the sense of Hausdorff dimension (see \cite[Proposition 6]{Peres-Shmerkin-2009} and \cite[Lemma 2.4]{Orponen-2012}). This enables us to complete the proof of Theorem \ref{thm:critical-value}.

\begin{proof}[Proof of Theorem \ref{thm:critical-value}]
  Write $s = s_{\F}$.
  By Lemma \ref{lem:finite-P} we have $\mathfrak{P}^{s}(\F)< +\f$, and hence $\mathfrak{s}(\F) \le s$.
  It follows from Theorem \ref{thm:P-s} (ii) that for any sequence of periodic orbits $\{\mathcal O_k\}_{k=1}^\f$ in $(E_\F, |\cdot|, T_\F)$,
  \[
  \limsup_{k\to+\f}\#\mathcal O_k\cdot\eta(\mathcal O_k)^{s}<+\f.
  \]
  It suffices to show the lower bound $\mathfrak{s}(\F) \ge s$. This is clear if $s=0$. In the following we assume $s>0$.
  By Lemma \ref{lem:pi-bi-Lipschitz} and Corollary \ref{cor:bi-Lipschitz} it follows that $ \mathfrak{s}(\F) = \mathfrak{s}(\Sigma, d_{\F}, \sigma)$.
  Thus, we only need to show that $\mathfrak{s}(\Sigma, d_{\F}, \sigma) \ge s$.

  When $\F$ is homogeneous, i.e., $r_i = r$ for all $1 \le i \le m$, we have $s = - \log m / \log r$.
  Let $\{\ii^{(n)}\}\subset\Sigma$ be the sequence of periodic points constructed as in Proposition \ref{prop:homogeneous}.
  Then we have \[ \# \mathcal{O}_{\ii^{(n)}} \cdot   \eta( \mathcal{O}_{\ii^{(n)}} )  ^s = m^n \cdot (r^{n-1})^{-\log m / \log r} =m. \]
  By Theorem \ref{thm:P-s} (i), we have $\mathfrak{P}^s(\Sigma, d_{\F}, \sigma)>0$, and hence $\mathfrak{s}(\Sigma, d_{\F}, \sigma) \ge s$.

  Next, we consider the inhomogeneous IFS $\F$.
  Without loss of generality, we assume that $r_1 \le r_2 \le \ldots \le r_m$.
  Fix $0< \ep < s$, and choose a sufficiently large integer $N$ {such that
  \begin{equation}\label{eq:july13-1}
  \frac{m\log(N+1)}{-N\log r_m}<\ep.
  \end{equation}}
  Note that $\sum_{i=1}^{m}r_i^s=1$. This implies
  $$1=\bigg( \sum_{i=1}^m r_i^s \bigg)^N= \sum_{\ell_1+\cdots+\ell_m=N}\frac{N!}{\ell_1!\cdots \ell_m!}\prod_{i=1}^m r_i^{s \ell_i},$$
  where the sum is over all non-negative integers $\ell_1, \ldots, \ell_m$ with $\ell_1+\cdots+\ell_m=N$.
  Note that the number of non-negative integer solutions to the equation $\ell_1+\cdots+\ell_m=N$ is less than $(N+1)^m$.
  Thus there exist non-negative integers $\ell_1^*,\ldots,\ell_m^*$ with $\ell_1^*+\cdots+\ell_m^*=N$ such that
  \begin{equation}\label{dim_estimate_1}
	\frac{N!}{\ell_1^*!\cdots \ell^*_m!} \prod_{i=1}^mr_i^{s\ell^*_i}  \ge \frac{1}{(N+1)^m}.
  \end{equation}
  For $i_1 i_2 \ldots i_N \in \{1,2,\ldots,m\}^N$ and $1 \le j \le m$, let $L_j(i_1 i_2 \ldots i_N) := \# \{ 1\le k \le N: i_k = j\}$. Then $L_j(i_1\ldots i_N)$ is the number of digit $j$ appearing in the block $i_1\ldots i_N$.
  Define \[ \mathcal{G} = \big\{ f_{i_1 i_2 \ldots i_N} : L_j(i_1 i_2 \ldots i_N)=\ell_j^* \;\;\forall 1 \le j \le m \big\}. \]
  Then $\mathcal{G}$ is a homogeneous self-similar IFS with the common contractive ratio $r_{\mathcal{G}}=\prod_{i=1}^m r_i^{\ell_i^*}$, and $\mathcal{G}$ satisfies the SSC.
  Write $\mathcal{G}=\big\{ f_{\u^{(1)}}, f_{\u^{(2)}}, \ldots, f_{\u^{(q)}} \big\}$, where $\u^{(i)} \in \{1,2,\ldots,m\}^N$ and \[q = {\# \mathcal{G}} = \frac{N!}{\ell_1^*!\cdots \ell_m^*!}. \]
  Let $t$ denote the similarity dimension of $\mathcal{G}$, i.e., \[ \frac{N!}{\ell_1^*!\cdots \ell_m^*!}\bigg( \prod_{i=1}^m r_i^{\ell_i^*} \bigg)^{t} = 1. \]
  Together with (\ref{dim_estimate_1}), we obtain that
  \[ \frac{1}{(N+1)^m} \le \bigg( \prod_{i=1}^m r_i^{\ell_i^*} \bigg)^{ s-t} \le r_m^{N(s-t)}.  \]
  This together with (\ref{eq:july13-1}) implies
  \begin{equation}\label{eq:july13-2} t \ge s + \frac{m\log(N+1)}{N\log r_m}>s-\ep. \end{equation}

  Now we consider the symbolic dynamical system $(\widetilde{\Sigma}, d_{\mathcal{G}}, \widetilde{\sigma})$ associated with the homogeneous self-similar IFS $\mathcal{G}$, where $\widetilde{\Sigma} =\{ \u^{(1)},\u^{(2)}, \ldots, \u^{(q)} \}^\N$ and $\widetilde{\sigma}$ is the left shift on $\widetilde{\Sigma}$. Note that each digit in $\widetilde{\Sigma}$ is a block in $\set{1,2,\ldots,m}^N$.
  By Proposition \ref{prop:homogeneous}, for $n\in\N$ there exists a periodic point $\i^{(n)} = \u^{(i_1)}\u^{(i_2)}\ldots$ in $(\widetilde{\Sigma}, d_{\mathcal{G}}, \widetilde{\sigma})$ such that
  \begin{equation}\label{eq:july13-3}
   \#\mathcal{O}_{\ii^{(n)}} = q^n \quad \text{and} \quad \eta( \mathcal{O}_{\ii^{(n)}}) = r_{\mathcal{G}}^{n-1}. \end{equation}
 Based on $\ii^{(n)}$ we construct a periodic point $\jj^{(n)}$ in the original symbolic dynamical system $(\Sigma, d_{\F},\sigma)$ associated with the self-similar IFS $\F$ by
   \begin{equation}\label{eq:july13-4}\jj^{(n)} := \underbrace{1^{2N-1} m\; \u^{(i_1)} \ldots \u^{(i_n)}}_{\textrm{length}\;N(n+2)}\; \underbrace{1^{2N-1} m\; \u^{(i_{n+1})} \ldots \u^{(i_{2n})}}_{\textrm{length}\;N(n+2)}\; \ldots.\end{equation}
  We claim that $\jj^{(n)}$ is a periodic point in $(\Sigma, d_{\F},\sigma)$ with
  \begin{equation}\label{eq:july13-5} \#\mathcal{O}_{\jj^{(n)}} = N \cdot (n+2) \cdot \frac{q^n}{\gcd(q^n,n)} \ge N q^n \quad \text{and}\quad \eta(\mathcal{O}_{\jj^{(n)}}) \ge r_1^{4N} r_{\mathcal{G}}^{n+1}. \end{equation}
  If the claim were proved, then noting that $t = - \log q / \log r_{\mathcal{G}}$ we have
  \[ \#\mathcal{O}_{\jj^{(n)}} \cdot  \eta(\mathcal{O}_{\jj^{(n)}})  ^t \ge Nq^n \cdot ( r_1^{4N} r_{\mathcal{G}}^{n+1} )^t  =  N(r_1^{4N} r_{\mathcal{G}})^t. \]
  Therefore, we obtain a sequence of periodic points $\{\jj^{(n)}\}_{n=1}^\f$ in $(\Sigma, d_{\F}, \sigma)$ such that  \[ \lim_{n \to +\f} \# \mathcal{O}_{\jj^{(n)}} = +\f \quad \text{and}\quad
  \liminf_{n \to +\f} \#\mathcal{O}_{\jj^{(n)}} \cdot   \eta(\mathcal{O}_{\jj^{(n)}})  ^t \ge  N(r_1^{4N} r_{\mathcal{G}})^t >0. \]
  By Theorem \ref{thm:P-s} (i), we obtain that $\mathfrak{P}^t(\Sigma, d_{\F},\sigma) > 0$. By (\ref{eq:july13-2}) it follows that $\mathfrak{s}(\Sigma, d_{\F},\sigma) \ge t > s-\ep$.
  Letting $\ep \to 0$, we conclude that $\mathfrak{s}(\Sigma, d_{\F},\sigma) \ge s$.

  It remains to prove the claim.
  Since $r_1 \le r_m$ and $\ep < s$, by (\ref{eq:july13-1}) and (\ref{dim_estimate_1}) we have $\ell_1^* < N$.
  For any $1\le i \le q$, we have $L_1(\u^{(i)}) = \ell_1^* < N$ and hence, $\u^{(i)} \ne 1^N$.
  First we consider $\#\mathcal O_{\jj^{(n)}}$ in (\ref{eq:july13-5}). Suppose that $\sigma^k(\jj^{(n)}) = \jj^{(n)}$ for some $k \in \N$. Note that $\u^{(i_\ell)} \ne 1^N$ for all $\ell \in \N$. Then by (\ref{eq:july13-4}) it follows that
  \[k=N\cdot (n+2) \cdot \zeta\quad \textrm{for some}\;\; \zeta \in \set{1,2,\ldots,q^n},\]
  which implies that $\widetilde{\sigma}^{\zeta n}(\ii^{(n)}) = \ii^{(n)}$. Note by (\ref{eq:july13-3}) that
    $q^n$ is the least period of $\ii^{(n)}$. So we have $q^n \mid \zeta n$, i.e.,
  \[ \frac{q^n}{\gcd(q^n, n)} \mid \zeta. \]
  Thus,   the least period of $\jj^{(n)}$ is given by   \[\#\mathcal{O}_{\jj^{(n)}}  = N \cdot (n+2) \cdot \frac{q^n}{\gcd(q^n,n)}\ge N q^n. \]

Next we consider $\eta(\mathcal O_{\jj^{(n)}})$ in (\ref{eq:july13-5}).  For $0\le k_1 < k_2 < \#\mathcal{O}_{\jj^{(n)}}$, write
\begin{equation}\label{eq:k-i}
k_1 = \zeta_1 N(n+2) + \tau_1\quad \textrm{and}\quad k_2 = \zeta_2 N(n+2) + \tau_2,
\end{equation} where $0\le \zeta_1 \le \zeta_2 < q^n / \gcd(q^n,n)$, $0\le \tau_1, \tau_2 < N(n+2)$.
  Suppose  on the contrary that
  \begin{equation}\label{eq:d-F-small}
    d_{\F}\Big( \sigma^{k_1}(\jj^{(n)}), \sigma^{k_2}(\jj^{(n)}) \Big) < r_1^{4N} r_{\mathcal{G}}^{n+1}.
  \end{equation}
  Let $v_1 v_2 \ldots v_\ell\in\set{1,2,\ldots,m}^*$ be the longest common prefix of $\sigma^{k_1}(\jj^{(n)})$ and $\sigma^{k_2}(\jj^{(n)})$.
  Note that $r_{\mathcal{G}} \ge r_1^N$.
  By (\ref{eq:d-F-small}), in view of the forms of $\jj^{(n)}$ defined in (\ref{eq:july13-4}) we have $\ell > N(n+4)$.
  This means that the finite word $1^{2N-1}m$ appears in the common prefix $v_1 v_2 \ldots v_\ell$.
  Note that $\u^{(i_\ell)} \ne 1^N$ for all $\ell \in \N$.
  So we must have $\tau_1 = \tau_2 = \tau$.
  If $0\le \tau \le 2N$, then by (\ref{eq:july13-4}) and (\ref{eq:k-i}) it follows that
  \begin{align*}
    \si^{k_1}(\jj^{(n)}) & =\si^{\tau}(1^{2N-1}m)\underbrace{\u^{(i_{\zeta_1 n+1})} \u^{(i_{\zeta_1 n+2})} \ldots\u^{(i_{\zeta_1 n+n})}\;1^{2N-1}m}_{\textrm{length}\;N(n+2)}\ldots; \\
    \si^{k_2}(\jj^{(n)}) & =\si^{\tau}(1^{2N-1}m)\underbrace{\u^{(i_{\zeta_2 n+1})} \u^{(i_{\zeta_2 n+2})} \ldots\u^{(i_{\zeta_2 n+n})}\;1^{2N-1}m}_{\textrm{length}\; N(n+2)}\ldots.
  \end{align*}
  So, by (\ref{eq:d-F-small}) we have
  \[ \u^{(i_{\zeta_1 n +1})} \u^{(i_{\zeta_1 n +2})} \ldots \u^{(i_{\zeta_1 n+n})} = \u^{(i_{\zeta_2 n +1})} \u^{(i_{\zeta_2 n +2})} \ldots \u^{(i_{\zeta_2 n+n})}. \]
  This leads to a contradiction with $\eta(\mathcal{O}_{\ii^{(n)}}) = r_{\mathcal{G}}^{n-1}$.
  If $2N<\tau < N(n+2)$, then write $\tau = (2+\xi)N + \eta$ with $0\le \xi< n$ and $0\le \eta <N$, and by (\ref{eq:july13-4}) and (\ref{eq:k-i}) it follows that
  \begin{align*}
    \si^{k_1}(\jj^{(n)}) & =\si^{\eta}(\u^{(i_{\zeta_1 n+\xi+1})})\underbrace{\u^{(i_{\zeta_1 n+\xi+2})}\ldots\u^{(i_{\zeta_1 n+n})}\;1^{2N-1}m\; \u^{(i_{\zeta_1 n+n+1})}\ldots\u^{(i_{\zeta_1 n+n+\xi+1})}}_{\textrm{length}\;N(n+2)}\ldots; \\
    \si^{k_2}(\jj^{(n)}) & =\si^{\eta}(\u^{(i_{\zeta_2 n+\xi+1})})\underbrace{\u^{(i_{\zeta_2 n+\xi+2})}\ldots\u^{(i_{\zeta_2 n+n})}\;1^{2N-1}m\; \u^{(i_{\zeta_2 n+n+1})}\ldots\u^{(i_{\zeta_2 n+n+\xi+1})}}_{\textrm{length}\;N(n+2)}\ldots.
  \end{align*}
  Thus, by (\ref{eq:d-F-small}) we have
  \[ \u^{(i_{\zeta_1 n +\xi +2})} \u^{(i_{\zeta_1 n +\xi +3})} \ldots \u^{(i_{\zeta_1 n +\xi+n+1})} = \u^{(i_{\zeta_2 n +\xi +2})} \u^{(i_{\zeta_2 n +\xi +3})} \ldots \u^{(i_{\zeta_2 n +\xi+n+1})}. \]
  This also leads to a contradiction with $\eta(\mathcal{O}_{\ii^{(n)}}) = r_{\mathcal{G}}^{n-1}$.
  We have proved the claim. The proof is complete.
\end{proof}

\section{Approximating self-similar measure by periodic measures}\label{sec:self-similar-measure}

In this section, we will prove Theorem \ref{thm:periodic-measure}.
We always assume that $\mathcal{F} =\{f_i(x) = r_i O_i x + b_i\}_{i=1}^m$ is a self-similar IFS on $\R^d$ that satisfies the SSC, and write $s = s_{\F}$.
Recall the symbolic dynamical system $(\Sigma, d_{\F}, \sigma)$ and the coding map $\pi$ defined in Section \ref{sec:self-similar-IFS}.

Given a probability vector $\p=(p_1, \ldots, p_m)$, let $\nu_{\p}$ be the Bernoulli measure on $\Sigma$ associated with the probability vector $\p$.
Then the self-similar measure $\mu_{\p}$ is the projection of $\nu_{\p}$, i.e.,
\[\mu_{\p} = \nu_{\p} \circ \pi^{-1}.\]
Write $\Omega = \{1,2,\ldots,m\}$.
For a finite word $i_1 i_2 \ldots i_n \in \Omega^n$, define the level-$n$ cylinder
\[ [i_1 i_2 \ldots i_n] := \big\{ (j_k)_{k=1}^\f \in \Sigma: j_1 j_2 \ldots j_n = i_1 i_2 \ldots i_n \big\}. \]
For $n \in \N$, let $\mathscr{P}^n= \big\{ [i_1 i_2 \ldots i_n]: i_1 i_2 \ldots i_n \in \Omega^n \big\}$ denote the partition consisting of all level-$n$ cylinders.

Let $\mathscr{A}=\{A_1,A_2,\dots,A_n\}$ be a finite measurable partition of the measurable space $(X,\mathscr{B})$, and let $\mu,\nu$ be two probability measures on $(X,\mathscr{B})$.
The \emph{relative entropy} or \emph{Kullback–Leibler divergence} with respect to $\mathscr{A}$ is defined by
\[ D_{\mathrm{KL}}^{\mathscr{A}}(\mu\|\nu):=\sum_{i=1}^n\mu(A_i)\log\frac{\mu(A_i)}{\nu(A_i)},\]
where we use the convention that $0 \log \frac{0}{\;\cdot\;} =0$.
It is well-known that $D_{\mathrm{KL}}^{\mathscr{A}}(\mu\|\nu) \ge 0$ (cf. \cite[Theorem 2.6.3]{Cover-Thomas-2006}).
If a finite measurable partition $\mathscr{C}$ refines $\mathscr{A}$, i.e., any element in $\mathscr{A}$ is a union of finite elements in $\mathscr{C}$, then applying the log sum inequality (cf. \cite[Theorem 2.7.1]{Cover-Thomas-2006}) we have
\begin{align*}
  D_{\mathrm{KL}}^{\mathscr{C}}(\mu\|\nu)
  & = \sum_{C \in \mathscr{C}} \mu(C)\log\frac{\mu(C)}{\nu(C)} \\
  & = \sum_{A \in \mathscr{A}} \sum_{C \in \mathscr{C}_A} \mu(C)\log\frac{\mu(C)}{\nu(C)} \\
  & \ge \sum_{A \in \mathscr{A}} \mu(A)\log\frac{\mu(A)}{\nu(A)} \\
  & = D_{\mathrm{KL}}^{\mathscr{A}}(\mu\|\nu),
\end{align*}
where $\mathscr{C}_A = \{ C \in \mathscr{C}: C \subset A \}$ for $A\in \mathscr{A}$.
The Pinsker's inequality (cf. \cite[Lemma 11.6.1]{Cover-Thomas-2006}) gives that \[ \sup_{A \in \mathscr{A}} |\mu(A)-\nu(A)| \le \sqrt{\frac{1}{2} D_{\mathrm{KL}}^{\mathscr{A}}(\mu\|\nu)}. \]

\begin{lemma}\label{lem:KL-convergence}
Let $\nu = \nu_{\p}$ be the Bernoulli measure on $\Sigma$.
Then for any $\sigma$-invariant probability measure $\mu$ on $\Sigma$, we have
\[ D_{\mathrm{KL}}^{\mathscr{P}^{kn}}(\mu\|\nu)\geq k \cdot D_{\mathrm{KL}}^{\mathscr{P}^n}(\mu\|\nu)\quad \forall k\in\mathbb{N},\;\forall n\in\mathbb{N}.
\]
\end{lemma}
\begin{proof}
Fix $k,n \in \N$.
Let $\lambda$ be a probability measure on $\Sigma$ such that
$$\lambda([\u_1 \u_2 \ldots \u_\ell])= \prod_{i=1}^\ell \mu([\u_i]) \quad \forall \ell \in \N,\; \forall \u_1, \u_2, \ldots, \u_\ell \in \Omega^n.$$
We have
\begin{align*}
D_{\mathrm{KL}}^{\mathscr{P}^{kn}}(\mu\|\nu)
& = \sum_{\u_1, \ldots, \u_k \in \Omega^n}\mu([\u_1\u_2\ldots \u_k])\log \frac{\mu([\u_1\u_2 \ldots \u_k])}{\nu([\u_1\u_2\ldots \u_k])} \\
& = \sum_{\u_1, \ldots, \u_k \in \Omega^n}\mu([\u_1 \u_2\ldots \u_k])\left( \log \frac{\mu([\u_1 \u_2 \cdots \u_k])}{\lambda([\u_1\u_2\ldots \u_k])} + \log \frac{\lambda([\u_1\u_2\ldots \u_k])}{\nu([\u_1\u_2 \ldots \u_k])} \right) \\
&=D_{\mathrm{KL}}^{\mathscr{P}^{kn}}(\mu\|\lambda) + \sum_{\u_1, \ldots, \u_k \in \Omega^n}\mu([\u_1 \u_2\ldots \u_k]) \sum_{i=1}^k\log \frac{\mu([\u_i])}{\nu([\u_i])}\\
& \ge \sum_{\u_1, \ldots, \u_k \in \Omega^n}\mu([\u_1 \u_2\ldots \u_k]) \sum_{i=1}^k\log \frac{\mu([\u_i])}{\nu([\u_i])} \\
& = \sum_{i=1}^k \sum_{\u \in \Omega^n} \log \frac{\mu([\u])}{\nu([\u])} \sum_{\u_1,\ldots, \u_{i-1}, \u_{i+1},\ldots, \u_k \in \Omega^n} \mu([\u_1\ldots \u_{i-1} \u \u_{i+1} \ldots \u_k])\\
& = \sum_{i=1}^k \sum_{\u \in \Omega^n} \log \frac{\mu([\u])}{\nu([\u])} \sum_{\u_1,\ldots, \u_{i-1} \in \Omega^n} \mu([\u_1\ldots \u_{i-1} \u])\\
& = \sum_{i=1}^k \sum_{\u \in \Omega^n} \mu(\sigma^{-(i-1)n} [\u]) \log \frac{\mu([\u])}{\nu([\u])} \\
& = \sum_{i=1}^k \sum_{\u \in \Omega^n} \mu([\u])\log \frac{\mu([\u])}{\nu([\u])} \\
&=k \cdot D_{\mathrm{KL}}^{\mathscr{P}^{n}}(\mu\|\nu),
\end{align*}
where the penultimate equality follows from the $\sigma$-invariance of $\mu$.
\end{proof}
We will prove Theorem \ref{thm:periodic-measure} by passing to the symbolic dynamical system $(\Sigma, d_{\F}, \sigma)$.

\begin{proof}[Proof of Theorem \ref{thm:periodic-measure}]
  Write $\mu = \mu_{\p}$ and $\nu= \nu_{\p}$ for $\p=(r_1^s, r_2^s, \ldots, r_m^s)$ with $s=s_{\F}$.
  Note that $\{x_n\}_{n=1}^\f$ is a sequence of periodic points in $(E_{\F}, |\cdot|, T_{\F})$ with $\#\mathcal{O}_{x_n} \to +\f$ as $n \to +\f$, and $\{\mu_n\}_{n=1}^\f$ is the sequence of corresponding periodic measures.
  Let $\ii^{(n)} = \pi^{-1}(x_n)\in \Sigma$.
  Then $\{\ii^{(n)}\}_{n=1}^\f$ is a sequence of periodic points in $(\Sigma, d_{\F}, \sigma)$ with $\#\mathcal{O}_{\ii^{(n)}} = \# \mathcal{O}_{x_n}$.
  Note by Lemma \ref{lem:pi-bi-Lipschitz} that $\pi$ is bi-Lipschitz.
  Then by the assumption we have
  \begin{equation}\label{eq:liminf-ge-s}
    \liminf_{n \to +\f} \frac{\log \#\mathcal{O}_{\ii^{(n)}}}{-\log \eta(\mathcal{O}_{\ii^{(n)}}) } \ge  s
  \end{equation}
  Consider the periodic probability measures \[ \nu_n = \frac{1}{ \#\mathcal{O}_{\ii^{(n)}} } \sum_{\boldsymbol{y} \in \mathcal{O}_{\ii^{(n)}}} \delta_{\boldsymbol{y}}, \]
  and then we have $\mu_n = \nu_n \circ \pi^{-1}$.
  Note that $\mu = \nu \circ \pi^{-1}$.
  It suffices to show that $\{\nu_n\}_{n=1}^\f$ converges weakly to $\nu$.
  Equivalently, we need to show that for any $k \in \N$ and for any $\u \in \Omega^k$,
  \[ \lim_{n \to +\f} \nu_n([\u]) = \nu([\u]).  \]
  In the following, fix $k \in \N$. Applying the Pinsker's inequality, we have
  \[ \sup_{\u \in \Omega^k}|\nu_n([\u]) - \nu([\u])| \le \sqrt{\frac{1}{2} D_{\mathrm{KL}}^{\mathscr{P}^k}(\nu_n\|\nu)}. \]
  It remains to show \[ \lim_{n \to +\f} D_{\mathrm{KL}}^{\mathscr{P}^k}(\nu_n\|\nu) =0.  \]

  Let $r = \min\{ r_1, r_2, \ldots, r_m\}$ and write $N_n=\#\mathcal{O}_{\ii^{(n)}}$ and $\eta_n=\eta(\mathcal{O}_{\ii^{(n)}})$.
  Define the partition \[ \mathscr{C}_n=\big\{ [u_1 u_2 \ldots u_\ell]: r_{u_1} r_{u_2} \cdots r_{u_\ell}<\eta_n \leq r_{u_1} r_{u_2} \cdots r_{u_{\ell-1}},\; \ell \in \N \big\}. \]
  For any $[u_1 u_2 \ldots u_\ell] \in \mathscr{C}_n$, since $r_{u_1} r_{u_2} \cdots r_{u_\ell}<\eta_n$, the cylinder $[u_1 u_2 \ldots u_\ell]$ contains at most one point in $\mathcal{O}_{\ii^{(n)}}$, and hence, $\nu_n([u_1 u_2 \ldots u_\ell]) = \frac{1}{N_n}$ or $\nu_n([u_1 u_2 \ldots u_\ell]) = 0$.
  Note that $\nu([u_1 u_2 \ldots u_\ell]) = (r_{u_1} r_{u_2} \cdots r_{u_\ell})^s \ge (r\eta_n)^s$ for any $[u_1 u_2 \ldots u_\ell] \in \mathscr{C}_n$.
  Thus, we have
  \begin{equation}\label{eq:KL-upper-bound}
  \begin{split}
    D_{\mathrm{KL}}^{\mathscr{C}_{n}}(\nu_n\|\nu)
    &=\sum_{[\u]\in \mathscr{C}_n,\nu_n([\u])>0} \nu_n([\u])\log \frac{\nu_n([\u])}{\nu([\u])}\\
	&\le \sum_{[\u]\in \mathscr{C}_n,\nu_n([\u])>0} \frac{1}{N_n}\log \frac{1}{(r\eta_n)^sN_n}\\
	& = -\log N_n - s \log \eta_n - s \log r.
  \end{split}
  \end{equation}

  For any $[u_1 u_2 \ldots u_\ell] \in \mathscr{C}_n$, since $r_{u_1} r_{u_2} \ldots r_{u_\ell}<\eta_n$, we have $\ell > \log \eta_n / \log r$.
  Note that $k$ is fixed.
  Then for all sufficiently large $n$, we can choose $\ell_n \in \N$ such that
  \begin{equation}\label{eq:ell-n}
    \ell_n k \le \log \eta_n / \log r < 2\ell_n k.
  \end{equation}
  Then the partition $\mathscr{C}_n$ refines $\mathscr{P}^{\ell_n k}$.
  Note that $\nu_n$ is $\sigma$-invariant. By Lemma \ref{lem:KL-convergence}, we have \[ D_{\mathrm{KL}}^{\mathscr{C}_{n}}(\nu_n\|\nu) \ge D_{\mathrm{KL}}^{\mathscr{P}^{\ell_n k}}(\nu_n\|\nu) \ge \ell_n \cdot D_{\mathrm{KL}}^{\mathscr{P}^{k}}(\nu_n\|\nu). \]
  Together with (\ref{eq:KL-upper-bound}) and (\ref{eq:ell-n}), we obtain
  \[ D_{\mathrm{KL}}^{\mathscr{P}^{k}}(\nu_n\|\nu) \le \frac{-\log N_n - s \log \eta_n - s \log r}{ \ell_n } \le 2 k (-\log r) \left( s+ \frac{\log N_n}{\log \eta_n} + \frac{s \log r}{\log \eta_n} \right).  \]
  Thus, (\ref{eq:liminf-ge-s}) implies that \[ \lim_{n \to +\f} D_{\mathrm{KL}}^{\mathscr{P}^k}(\nu_n\|\nu) =0, \] as desired
\end{proof}
\begin{remark}
Under a stronger condition than (\ref{eq:s-F-lower-bound}), the argument in the proof of Theorem \ref{thm:periodic-measure} yields a quantitative version of the weak convergence.
If \[ \# \mathcal{O}_{x_n}\cdot \eta(\mathcal{O}_{x_n})^s \ge c >0 \quad \forall n \in \N, \]
then there exists a constant $C>0$, depending only on $\F$ and $c$, such that \[ \sup_{\u \in \Omega^k}|\nu_n([\u]) - \nu([\u])| \le C \sqrt{ \frac{k}{-\log \eta(\mathcal{O}_{x_n})} }. \]
\end{remark}

\section{Final remarks}\label{sec:remarks}

Suppose that $\mathcal{F} =\big\{ f_i(x) = r_i O_i x + b_i \big\}_{i=1}^m$ is a self-similar IFS on $\R^d$ satisfying the SSC.
By Theorem \ref{thm:critical-value}, we have $\mathfrak{s}(E_\F, |\cdot|, T_\F)=s_\F$.
It is natural to ask whether the $s_{\F}$-dimensional periodic content $\mathfrak{P}^{s_\F}(E_\F, |\cdot|, T_\F)$ is positive and finite.
The finiteness of $\mathfrak{P}^{s_\F}(E_\F, |\cdot|, T_\F)$ follows from Lemma \ref{lem:finite-P}.
When $\mathcal{F}$ is homogeneous, i.e., $r_i = r$ for all $1 \le i \le m$, Proposition \ref{prop:homogeneous} together with Theorem \ref{thm:P-s} (i) implies that $\mathfrak{P}^{s_\F}(E_\F, |\cdot|, T_\F)>0$.
However, when $\mathcal{F}$ is inhomogeneous, the proof of Theorem \ref{thm:critical-value} does not yield the positivity of $\mathfrak{P}^{s_\F}(E_\F, |\cdot|, T_\F)$.
This leads to the following question.

\begin{question}
  Suppose that $\mathcal{F} =\{f_i(x) = r_i O_i x + b_i\}_{i=1}^m$ is a self-similar IFS on $\R^d$ satisfying the SSC, and $r_i \ne r_j$ for some $1 \le i < j \le m$.
  Is the $s_{\F}$-dimensional periodic content $\mathfrak{P}^{s_\F}(E_\F, |\cdot|, T_\F)$ positive? In view of Theorem \ref{thm:P-s} (i), it is equivalent to ask whether there exists a sequence of periodic points $\{x_n\}$ in $(E_\F, |\cdot|, T_\F)$ such that
  \[ \lim_{n \to +\f}\# \mathcal{O}_{x_n} = +\f \quad \textrm{and} \quad \liminf_{n \to +\f} \# \mathcal{O}_{x_n} \cdot  \eta(\mathcal{O}_{x_n})^{s_{\F}} >0. \]
\end{question}

The result in (\ref{result-Yu-Gan}) of Gan and the third author can be reformulated as $\mathfrak{P}^d(\mathbb{T}^d, |\cdot|, T) >0$ for any ergodic endomorphism $T$ on the $d$-torus $\mathbb{T}^d$. Combined with Lemma \ref{lem:upper-bound}, this yields $\mathfrak{s}(\mathbb{T}^d, |\cdot|, T) = d$.
It is interesting to investigate the periodic dimension $\mathfrak{s}(X,d,T)$ for other dynamical systems, such as $\beta$-transformations on the unit circle, continued fraction dynamical system, or some other numeration dynamical systems. It would be also interesting to find the intrinsic connections between the periodic dimension and the complexity of the ambient dynamical system, such as topological entropy and  Lyapunov exponents.

\section*{Acknowledgements}
Kong was supported by Chongqing NSF: CQYC20220511052 and Scientific Research Innovation Capacity Support Project for Young Faculty No.~ZYGXQNISKYCXNLZCXM-P2P.
Wang was supported by the National Natural Science Foundation of China No.~12501110, 12471085.
Yu was supported by the Fundamental Research Funds for the Central Universities No.~2025CDJZKPT‑09, and the Chongqing Postdoctoral Retention Grant No.~2509013610520348.


\end{document}